\documentclass[11pt,reqno]{amsart}

\usepackage[T1]{fontenc}
\usepackage[utf8]{inputenc}
\usepackage{amsmath,amsfonts,amsthm,amssymb,amsxtra,bbm}
\usepackage{setspace,graphicx,color,pdflscape}
\usepackage[colorlinks=true]{hyperref}
\hypersetup{urlcolor=blue, linkcolor=blue, citecolor=red, anchorcolor=blue]}
\usepackage{doi}

\newtheorem{theorem}{Theorem}
\newtheorem{proposition}[theorem]{Proposition}
\newtheorem{lemma}[theorem]{Lemma}

\newcommand{\R}{{\mathbb R}}

\newcommand{\be}[1]{\begin{equation}\label{#1}}
\newcommand{\ee}{\end{equation}}
\renewcommand{\(}{\left(}
\renewcommand{\)}{\right)}
\newcommand{\irtwo}[1]{\int_{\R^2}{#1}\,dx}
\newcommand{\ird}[1]{\int_{\R^d}{#1}\,dx}
\newcommand{\irtwomu}[1]{\int_{\R^2}{#1}\,d\mu}
\newcommand{\nrm}[2]{\left\|{#1}\right\|_{#2}}

\newcommand{\param}{\tau}

\makeatletter\@namedef{subjclassname@2020}{\textup{2020} Mathematics Subject Classification}\makeatother

\newcommand{\nc}{\normalcolor}

\title[Logarithmic estimates for $2$-dimensional mean-field models]{Logarithmic estimates for mean-field models in dimension two and the Schr\"odinger-Poisson system}

\author[J.~Dolbeault]{Jean Dolbeault}
\address{\hspace*{-12pt}J.~Dolbeault: CEREMADE (CNRS UMR n$^\circ$ 7534), PSL university, Universit\'e Paris-Dauphine\\ Place de Lattre de Tassigny, 75775 Paris 16, France}
\email{\href{mailto:dolbeaul@ceremade.dauphine.fr}{dolbeaul@ceremade.dauphine.fr}}

\author[R.L.~Frank]{Rupert L.~Frank}
\address{\hspace*{-12pt}R.L.~Frank: Mathematisches Institut, Ludwig-Maximilans Universit\"at M\"unchen, Theresienstr. 39, 80333 M\"unchen, Germany, and Munich Center for Quantum Science and Technology, Schellingstr. 4, 80799 M\"unchen, Germany, and Mathematics 253-37, Caltech, Pasadena, CA 91125, USA}
\email{\href{mailto:r.frank@lmu.de}{r.frank@lmu.de}}

\author[L.~Jeanjean]{Louis Jeanjean}
\address{\hspace*{-12pt}L.~Jeanjean: Laboratoire de Math\'ematiques (CNRS UMR n$^\circ$ 6623), Universit\'e of Bourgogne Franche-Comt\'e, 25030 Besan\c con Cedex, France}
\email{\href{mailto:louis.jeanjean@univ-fcomte.fr}{louis.jeanjean@univ-fcomte.fr}}

\begin{document}
\begin{abstract} In dimension two, we investigate a free energy and the ground state energy of the Schr\"odinger-Poisson system coupled with a logarithmic nonlinearity in terms of underlying functional inequalities which take into account the scaling invariances of the problem. Such a system can be considered as a nonlinear Schr\"odinger equation with a cubic but nonlocal Poisson nonlinearity, and a local logarithmic nonlinearity. Both cases of repulsive and attractive forces are considered. We also assume that there is an external potential with minimal growth at infinity, which turns out to have a logarithmic growth. Our estimates rely on new logarithmic interpolation inequalities which combine logarithmic Hardy-Littlewood-Sobolev and logarithmic Sobolev inequalities. The two-dimensional model appears as a limit case of more classical problems in higher dimensions. \end{abstract}
\date{\today. {\em File:} \textsf{\jobname.tex}}
\maketitle
\thispagestyle{empty}
\vspace*{-0.75cm}

\section{The Schr\"odinger-Poisson system with a local logarithmic nonlinearity}

The standard Schr\"odinger-Poisson (SP) system is a nonlinear Schr\"odinger equation with cubic but nonlocal nonlinearity. As for the nonlinear Schr\"odinger (NLS) equation with a local nonlinearity, scaling properties play a crucial role in the analysis of the solutions and depend on the dimension $d$ of the Euclidean space. The fact that the nonlinearity in (SP) involves the Poisson convolution kernel makes existence results easier to study than for (NLS) because of the compactness properties induced by the convolution, but adds difficulties due to the non-locality of the mean field potential. We consider primarily the case $d=2$.

Our purpose is to focus on the underlying functional inequalities and study the interaction of the Poisson term with other terms in the energy (external potential, local nonlinearities) with similar scaling properties: we shall consider quantities which are all critical for (SP) in the two-dimensional case. This is quite interesting from the mathematical point of view, as it is a threshold case for (SP) systems and involves a non sign-defined logarithmic kernel. The $d=2$ case complements the results of~\cite{Cingolani_2019,Cingolani_2016} in the limit regime involving logarithmic local nonlinearities. For related questions for $d=3$, we refer to~\cite{Cingolani_2019} and references therein. In higher dimensions, the problem is sub-critical if $d\le5$ and critical for $d=6$: see Section~\ref{Sec:IneqHigher}.

The (SP) system is used in quantum mechanics to represent a large number of particles by a single complex valued wave function. The local nonlinear term arises from local effects or thermodynamical considerations while the non-local Poisson potential accounts for long range forces which are either of repulsive nature (charged particles) or attractive (in case of gravitational and related models). Most models in the physics literature are justified only on an empirical basis as thermodynamical limits but are difficult to establish rigorously. This issue is anyway out of the scope of this paper.

The Schr\"odinger equation with a \emph{logarithmic nonlinearity} is a remarkable model in physics, with interesting mathematical properties: see~\cite{MR426670,Carles_2018,L_pez_2011}. The equation has soliton-like solutions of Gaussian shape (called \emph{Gaussons} in~\cite{MR426670}). We shall refer to~\cite{MR583902,MR719365,MR2002047,MR3195154} for some additional contributions in mathematics. \emph{Schr\"odinger-Poisson} systems are commonly used in charged particles transport and particularly in semiconductor physics, in the repulsive case. In this direction, a classical reference for mathematical properties is~\cite{MR1487521} and we also quote~\cite{MR2038149,MR2034231} for examples of applications. The mean-field attractive case (Newton equation) reflects gravitational forces instead of electrostatic forces. It is not studied as much as the repulsive case and it is mathematically more difficult: see for instance~\cite[Section~4]{MR2034231}. As a side remark, we may notice that stationary solutions of (SP) share many properties with stationary solutions of two-dimensional models of chemotaxis, and the same functional inequalities are involved: see~\cite{2019arXiv190903667D}. We can however handle the two cases, attractive and repulsive, in a common framework. We primarily focus on variational results, in relation with some interesting functional inequalities and their scaling properties.\medskip

For any function $u\in\mathrm H^1(\R^2)$, let us consider the \emph{Schr\"odinger energy}
\be{SchEnergy}
\mathcal E[u]:=\irtwo{|\nabla u|^2}+\alpha\irtwo{V\,|u|^2}+2\,\pi\,\beta\irtwo{W\,|u|^2}+\gamma\irtwo{|u|^2\,\log|u|^2}
\ee
where $\alpha$, $\beta$, $\gamma$ are real parameters and the \emph{self-consistent potential} $W$ is obtained as a solution of the \emph{Poisson equation}
\[\label{Poisson}
-\,\Delta W=|u|^2\,.
\]
The solution $W$ of~\eqref{Poisson} is defined only up to an additive constant: we make the specific choice $W=(-\Delta)^{-1}|u|^2$ given by the Green kernel as follows. Let us recall that on $\R^2$ the standard Green function $G_y$ associated with $(-\Delta)$, that is, the solution of $-\Delta_xG=\delta_y(x)$, is given by
\[
G(x,y)=-\,\frac1{2\,\pi}\,\log|x-y|\quad\forall\,(x,y)\in\R^2\times\R^2\,.
\]
Our choice amounts to take $W(x)=\int_{\R^2}|u(y)|^2\,G(x,y)\,dy$. As a consequence, we have
\[
W(x)\sim-\,\frac{\nrm u2^2}{2\,\pi}\,\log|x|\quad\mbox{as}\quad |x|\to+\infty\,,
\]
and also $x\cdot\nabla W(x)<0$ for large values of $|x|$ if, for instance, $u$ is compactly supported. The cases $\beta>0$ and $\beta<0$ correspond to two very different physical situations. The case $\beta<0$ is the \emph{attractive} case of a Newton-Poisson coupling for gravitational mean-field models. With $\beta>0$, the model represents the two-dimensional case of \emph{repulsive} electrostatic forces, \emph{i.e.}, a mean field version of a quantum Coulomb gas of interacting particles in dimension $d=2$.

The function $V$ is an \emph{external potential}, and we shall assume that it has a \emph{critical growth}. The parameter $\alpha\in\R$ is a coupling parameter, whose value has to be discussed depending on the other terms. Without much loss of generality, we can assume that
\be{V}
V(x)=2\,\log\(1+|x|^2\)\quad\forall\,x\in\R^2\,.
\ee
Concerning the \emph{local nonlinearity}, the case $\gamma<0$ corresponds to a \emph{focusing local nonlinearity} while $\gamma>0$ is the case a \emph{defocusing local nonlinearity}. It is standard to observe that any critical point of $\mathcal E$ under the mass constraint
\[
\irtwo{|u|^2}=M
\]
determines a standing wave of the \emph{nonlinear Schr\"odinger-Poisson} system
\[
i\,\frac{\partial\psi}{\partial t}=\Delta\psi+\alpha\,V\,\psi+\beta\,W\,\psi+\,\gamma\,\log|\psi|^2\,\psi\,.
\]
In this paper we shall focus on finding conditions on $\alpha$, $\beta$, $\gamma\in\R$ insuring that the functional~$\mathcal{E}$ is either bounded or unbounded from below on
\[
\mathcal H_M:=\left\{u\in\mathrm H^1(\R^2)\,:\,\nrm u2^2=M\right\}\,.
\]

This paper is organized as follows. We establish in Section~\ref{Sec:Ineq} several new functional inequalities which generalize the logarithmic Hardy-Littlewood-Sobolev inequality, with an application to a \emph{free energy} functional in dimension two: see Theorem~\ref{Thm:FreeEnergy}. Section~\ref{Sec:Schrodinger} is devoted to the boundedness from below of the Schr\"odinger energy $\mathcal E$, with main results in Theorem~\ref{Thm:E}.

\section{New logarithmic inequalities and free energy estimates}\label{Sec:Ineq}

\subsection{Generalized logarithmic Hardy-Littlewood-Sobolev inequalities}

The \emph{logarithmic Hardy-Littlewood-Sobolev inequality}
\be{logHLS}
\irtwo{\rho\,\log\(\frac\rho M\)}+\frac 2M\iint_{\R^2\times\R^2}\rho(x)\,\rho(y)\,\log|x-y|\,dx\,dy+M\,(1+\log\pi)\ge0
\ee
has been established in optimal form in~\cite{MR1143664} by E.~Carlen and M.~Loss, and in~\cite{Beckner93} by W.~Beckner for any $\rho\in\mathrm L^1_+(\R^2)$ such that \hbox{$\irtwo\rho=M>0$}. Equality is achieved by $\rho=\rho_\star$ with
\be{rhostar}
\rho_\star(x):=\frac M{\pi\(1+|x|^2\)^2}\quad\forall\,x\in\R^2\,,
\ee
and also by any function obtained from $\rho_\star$ by a multiplication by a positive constant (with the corresponding mass constraint), a scaling or a translation. Alternative proofs based on fast diffusion flows have been obtained in~\cite{MR2745814,MR2915466,MR3227280}. Also see~\cite{MR1230930,MR2996772,Dolbeault:2015ly,MR677001} for further related results and considerations on dual Onofri type inequalities and \cite{MR2858468} for a rearrangement-free proof of~\eqref{logHLS} using reflection positivity. Inequality~\eqref{logHLS} provides us with a useful lower bound on the free energy in the case of an attractive Poisson equation corresponding to the Keller-Segel model: see~\cite{MR2226917,MR2103197}, or in the case of a mean-field Newton equation in gravitational models. In presence of the potential $V$ given by~\eqref{V}, we have
\begin{multline}\label{Ineq:LogHLS}
\irtwo{\rho\,\log\(\frac\rho M\)}+2\,\param\irtwo{\log\(1+|x|^2\)\rho}+M\,(1-\param+\log\pi)\\
\ge\frac 2M\,(\param-1)\iint_{\R^2\times\R^2}\rho(x)\,\rho(y)\,\log|x-y|\,dx\,dy
\end{multline}
for any $\param\ge0$ and for any function $\rho\in\mathrm L^1_+(\R^2)$ with $M=\irtwo\rho>0$, according to~\cite{2019arXiv190903667D}. Compared to~\cite{2019arXiv190903667D}, the discrepancy in the coefficient of $M$ in the last term of the r.h.s.~in~\eqref{Ineq:LogHLS} is due to the normalization of $V$ as defined by~\eqref{V}. Equality again holds if $\rho=\rho_\star$ given by~\eqref{rhostar}. When $\param=0$,~\eqref{Ineq:LogHLS} is nothing else than~\eqref{logHLS} while the case $\param=1$ is easily recovered by Jensen's inequality. Notice that the sign of the coefficient in front of the convolution term in the r.h.s.~of~\eqref{Ineq:LogHLS} becomes positive if $\param>1$.

Let us divide~\eqref{Ineq:LogHLS} by $\param>0$ and then take the limit as $\param\to+\infty$. By doing this, we obtain a new inequality, which differs from~\eqref{logHLS} and is of interest by itself.
\begin{lemma}\label{Lem:LogHLSlim} For any function $\rho\in\mathrm L^1_+(\R^2)$ such that $\irtwo\rho=M$, we have
\be{Ineq:LogHLS-l}
2\irtwo{\log\(1+|x|^2\)\rho}-M\ge \frac2M\iint_{\R^2\times\R^2}\rho(x)\,\rho(y)\,\log|x-y|\,dx\,dy\,.
\ee
Moreover equality in~\eqref{Ineq:LogHLS-l} is achieved if and only if $\rho=\rho_\star$.\end{lemma}
\begin{proof} We give a direct proof of~\eqref{Ineq:LogHLS-l}, which does not rely on~\eqref{Ineq:LogHLS}. A preliminary observation is that~\eqref{Ineq:LogHLS-l} makes sense, \emph{i.e.}, that
\[
\rho\mapsto\irtwo{\log\(1+|x|^2\)\rho}-\frac1M\iint_{\R^2\times\R^2}\rho(x)\,\rho(y)\,\log|x-y|\,dx\,dy
\]
is bounded from below. We may indeed notice that, for any $x$, $y\in\R^d$,
\[
|x-y|^2=|x|^2+|y|^2-2\,x\cdot y\le|x|^2+|y|^2+\(1+|x|^2\,|y|^2\)=\(1+|x|^2\)\(1+|y|^2\)\,,
\]
so that, after multiplying by $\rho(x)\,\rho(y)$ and integrating with respect to $x$ and $y$, we obtain
\begin{align*}
&2\iint_{\R^2\times\R^2}\rho(x)\,\rho(y)\,\log|x-y|\,dx\,dy\\
&\le\iint_{\R^2\times\R^2}\rho(x)\,\rho(y)\,\Big(\log\(1+|x|^2\)+\log\(1+|y|^2\)\Big)\,dx\,dy\le2\,M\irtwo{\log\(1+|x|^2\)\rho}\,.
\end{align*}
As a consequence, the problem is reduced to proving that the largest constant $C$ such that
\[
2\irtwo{\log\(1+|x|^2\)\rho}-C\ge \frac2M\iint_{\R^2\times\R^2}\rho(x)\,\rho(y)\,\log|x-y|\,dx\,dy
\]
is $C=M$.

At heuristic level, if we admit that $\rho_\star$ realizes the equality case, this equality can be established as follows. The potential $V$ given by~\eqref{V} is such that $\mu_\star=\frac1\pi\,e^{-V}=\frac{\rho_\star}M$ is a probability measure and we have
\begin{equation*}
\Delta V=8\,\pi\,\mu_\star\,.
\end{equation*}
One can also check that
\[
(-\Delta)^{-1}\mu_\star:=-\frac1{2\,\pi}\int_{\R^2}\log|x-y|\,\mu_\star(y)\,dy=-\,\frac V{8\,\pi}=-\,\frac1{4\,\pi}\,\log\(1+|x|^2\)
\]
which requires a careful analysis of the integration constants. Indeed, in radial coordinates, by solving the ordinary differential equation
\[
\big(r\,V'\big)'=\frac{8\,r}{\(1+r^2\)}\,,\quad V'(0)=0\,,\quad V(0)=V_0\,,
\]
a couple of integrations shows that
\[
V'(r)=\frac1r\(\frac4{1+r^2}-4\)\quad\mbox{and}\quad V(r)-V_0=\int_0^r\frac{4\,s}{1+s^2}\,ds=2\,\log\(1+r^2\)\,,
\]
so that $8\,\pi\,(-\Delta)^{-1}\mu_\star=-(V+V_0)$ with $V_0=0$. Alternatively, a direct proof is obtained by observing that
\[
V_0=4\int_{\R^2}\log|y|\,\mu_\star(y)\,dy=8\int_0^{+\infty}\frac{r\,\log r}{\(1+r^2\)^2}\,dr=0\,,
\]
where the last equality is a consequence of the change of variables $r\mapsto1/r$. Taking into account the identity
\[
\irtwo{\log\(1+|x|^2\)\,\mu_\star(x)}=\int_0^{+\infty}\frac{2\,r\,\log\(1+r^2\)}{\(1+r^2\)^2}\,dr=1\,,
\]
this is consistent with the fact that~$\rho_\star=M\,\mu_\star$ corresponds to the equality case in~\eqref{logHLS}, according to~\cite{MR1143664}. Altogether, we have $C=M$, meaning that~\eqref{Ineq:LogHLS-l} is an equality if $\rho=\rho_\star$.

After these preliminary considerations, which are provided only for a better understanding of the functional framework, let us give a proof. With no loss of generality, we may assume that $M=1$ because of the $1$-homogeneity of~\eqref{Ineq:LogHLS-l}. Let us notice that
\begin{multline*}
2\irtwo{\log\(1+|x|^2\)\rho}-1-2\iint_{\R^2\times\R^2}\rho(x)\,\rho(y)\,\log|x-y|\,dx\,dy\\
=-\,2\iint_{\R^2\times\R^2}\big(\rho(x)-\mu_\star(x)\big)\,\big(\rho(y)-\mu_\star(y)\big)\,\log|x-y|\,dx\,dy.
\end{multline*}
We recover that the equality case in~\eqref{Ineq:LogHLS-l} is achieved if $\rho=\mu_\star$. With $W=-(-\Delta)^{-1}(\rho-\mu_\star)$, we obtain
\begin{multline*}
-\,2\iint_{\R^2\times\R^2}\big(\rho(x)-\mu_\star(x)\big)\,\big(\rho(y)-\mu_\star(y)\big)\,\log|x-y|\,dx\,dy\\
=4\,\pi\irtwo{(\rho-\mu_\star)\,(-\Delta)^{-1}(\rho-\mu_\star)}\\
=-\,4\,\pi\irtwo{(\Delta W)\,W}=4\,\pi\irtwo{|\nabla W|^2}\ge0\,,
\end{multline*}
where the last equality is obtained by a simple integration by parts. This can be done only because $\irtwo{(\rho-\mu_\star)}=0$, a necessary and sufficient condition to guarantee that $\nabla W$ is square integrable (for a proof, one has to study the behavior of the solution of the Poisson equation as $|x|\to+\infty$). At this point it is clear that $\irtwo{|\nabla W|^2}=0$ if and only if $\rho=\mu_\star$. The general case with an arbitrary $M>0$ is obtained by writing $\rho_\star=M\;\mu_\star$, which concludes the proof.
\end{proof}

The equality case in~\eqref{Ineq:LogHLS-l} is achieved among radial functions. It is classical that the l.h.s.~is decreasing under symmetric decreasing rearrangements, while the r.h.s. is increasing. The strict rearrangement inequality for the logarithmic kernel is proved in~\cite[Lemma~2]{MR1143664}. As a limit case of $\iint_{\R^2\times\R^2}\big(\rho(x)-\mu_\star(x)\big)\,\big(\rho(y)-\mu_\star(y)\big)\,|x-y|^\lambda\,dx\,dy$ when $\lambda\to 0_-$, according to~\cite[Theorem~4.3]{MR1817225} (also see~\cite{lopes2019uniqueness} for interesting consequences), this is indeed expected. Justifying the square integrability of $\nabla W$ has therefore to be done only among radial functions, which is elementary using, \emph{e.g.}, a compactly supported function $\rho$ and a density argument.

Also notice that one can now recover~\eqref{Ineq:LogHLS} as a simple consequence of~\eqref{logHLS} and~\eqref{Ineq:LogHLS-l}. Next, we turn our attention to an inequality which is a consequence of convexity and Jensen's inequality.~Let
\[
\mathsf J_\eta[\rho]:=\irtwo{\rho\,\log\(\frac\rho{\nrm\rho1}\)}+\eta\irtwo{\log\(1+|x|^2\)\rho}\quad\forall\,\rho\in\mathrm L^1_+(\R^2)\,.
\]
\begin{lemma}\label{Lem:LogHLSlim1} Let $\eta>0$, $M>0$ and $\mathcal X_M:=\left\{\rho\in\mathrm L^1_+(\R^2)\,:\,\nrm\rho1=M\right\}$.
\begin{enumerate}
\item[(i)] If $\eta>1$, then $\mathsf J_\eta$ is bounded from below on $\mathcal X_M$ and
\be{Ineq:LogHLS-ll}
\irtwo{\rho\,\log\(\frac\rho M\)}+\eta\irtwo{\log\(1+|x|^2\)\rho}\ge M\,\log\Big(\frac{\eta-1}\pi\Big)\quad\forall\,\rho\in\mathcal X_M\,.
\ee
For any $\eta>1$, equality in~\eqref{Ineq:LogHLS-ll} is achieved by $\rho=M\,\rho_\eta$, where
\[
\rho_\eta(x):=\frac{\eta-1}{\pi\(1+|x|^2\)^\eta}\quad\forall\,x\in\R^2\,.
\]
\item[(ii)] If $\eta\in(0,1]$, then $\inf_{\mathcal X_M}\mathsf J_\eta=-\infty$.
\end{enumerate}
\end{lemma}
If $\eta=2$, then $\rho_2=\rho_\star$, while~\eqref{Ineq:LogHLS-ll} amounts to $\mathsf J_\eta[\rho]\ge\mathsf J_\eta[M\,\rho_\eta]$ for any $\eta>1$. If $\param$ is restricted to the range $[0,1]$, we notice as in~\cite{2019arXiv190903667D} that~\eqref{Ineq:LogHLS} is a simple convex combination, with coefficients $(1-\param)$ and $\param$, of~\eqref{logHLS} and~\eqref{Ineq:LogHLS-ll} written with $\eta=2$.
\begin{proof} A direct computation based on $\frac d{dr}\(1+r^2\)^{1-\eta}=-2\,(\eta-1)\,r\(1+r^2\)^{-\eta}$ shows that
\[
\irtwo{\rho_\eta}=2\,(\eta-1)\int_0^{+\infty}r\(1+r^2\)^{-\eta}\,dr=1
\]
for all $\eta>1$ and
\[
\mathsf J_\eta[\rho]=\irtwo{\rho\,\log\Big(\frac\rho{M\,\rho_\eta}\Big)}+M\,\log\Big(\frac{\eta-1}\pi\Big)\quad\forall\,\rho\in\mathcal X_M\,.
\]
Using that $u\mapsto u\,\log u-u+1$ is a convex function whose minimum is $0$, we get
\[
\irtwo{\rho\,\log\Big(\frac\rho{M\,\rho_\eta}\Big)}=\irtwo{\frac\rho{M\,\rho_\eta} \log\Big(\frac\rho{M\,\rho_\eta}\Big)\,M\,\rho_\eta}\ge\irtwo{\Big(\frac\rho{M\,\rho_\eta}-1\Big)\,M\,\rho_\eta}=0
\]
by taking $u=\rho/(M\,\rho_\eta)$ and then integrating against $M\,\rho_\eta\,dx$. This proves~\eqref{Ineq:LogHLS-ll} for any $\eta>1$, where equality holds as a consequence of $\mathsf J_\eta[M\,\rho_\eta]=M\,\log\big(\frac{\eta-1}\pi\big)$.

Let us consider the case $\eta\in(0,1]$ and take $\rho=M\,\rho_\zeta$ with $\zeta>1$ as a test function. With a few integrations by parts, we obtain
\begin{multline*}\label{Id-2}
\irtwo{\log\(1+|x|^2\)\,\rho_\zeta(x)}=2\,(\zeta-1)\int_0^{+\infty}r\,\log\(1+r^2\)\,\(1+r^2\)^{-\zeta}\,dr\\
=-\int_0^{+\infty}\frac d{dr}\(\(1+r^2\)^{1-\zeta}\)\log\(1+r^2\)\,dr=2\int_0^{+\infty}r\(1+r^2\)^{-\zeta}\,dr=\frac1{\zeta-1}\,,
\end{multline*}
\begin{multline*}
\irtwo{\rho_\zeta\,\log\rho_\zeta}\\
=\log\Big(\frac{\zeta-1}\pi\Big)\irtwo{\rho_\zeta}-\zeta\irtwo{\log\(1+|x|^2\)\,\rho_\zeta(x)}=\log\Big(\frac{\zeta-1}\pi\Big)-\frac\zeta{\zeta-1}\,,
\end{multline*}
so that $\lim_{\zeta\to1_+}\mathsf J_\eta[M\,\rho_\zeta]=-\infty$ because
\[
\frac 1M\,\mathsf J_\eta[M\,\rho_\zeta]=\irtwo{\rho_\zeta\,\log\rho_\zeta}+\eta\irtwo{\log\(1+|x|^2\)\,\rho_\zeta(x)}=\log\Big(\frac{\zeta-1}\pi\Big)-\frac{\zeta-\eta}{\zeta-1}\,.
\]
\end{proof}

\subsection{Boundedness from below of the free energy functional}

Let us consider the \emph{free energy} functional defined by
\[\label{F}
\mathcal F_{\mathsf a,\mathsf b}[\rho]:=\irtwo{\rho\,\log\(\frac\rho M\)}+\mathsf a\irtwo{\log\(1+|x|^2\)\rho}-\frac{\mathsf b}M\iint_{\R^2\times\R^2}\rho(x)\,\rho(y)\,\log|x-y|\,dx\,dy
\]
for any $\rho\in\mathrm L^1_+(\R^2)$ such that $\irtwo\rho=M$. We look for the range of the parameters $\mathsf a$ and~$\mathsf b$ such that
\be{Ineq:LogHLSimproved}
\mathcal F_{\mathsf a,\mathsf b}[\rho]\ge\mathcal C(\mathsf a,\mathsf b)\,M\quad\forall\,\rho\in\mathrm L^1_+(\R^2)\quad\mbox{such that}\quad\nrm\rho1=M,\,
\ee
for some constant $\mathcal C(\mathsf a,\mathsf b)$. Inequality~\eqref{Ineq:LogHLS} with $\param\ge0$ is obtained as the special case $\mathsf a=2\,\param$ and $\mathsf b=2\,(\param-1)$, with $\mathcal C(\mathsf a,\mathsf b)=M\,(\param-1-\log\pi)$, according to~\cite{2019arXiv190903667D}. As a consequence, we also know that~\eqref{Ineq:LogHLSimproved} holds for some $\mathcal C(\mathsf a,\mathsf b)>-\infty$ if $\mathsf a\ge2\,\param$ and $\mathsf b=2\,(\param-1)$, that is, $0\le\mathsf b+2\le\mathsf a$. This range can be improved. For instance, if $\mathsf b=0$, it is clear from Lemma~\ref{Lem:LogHLSlim1} that the threshold is at $\mathsf a=1$ and not $\mathsf a=2$. Our result (see Fig.~\ref{Fig0}) is as follows.
\begin{theorem}\label{Thm:FreeEnergy}  Inequality~\eqref{Ineq:LogHLSimproved} holds for some $\mathcal C(\mathsf a,\mathsf b)>-\infty$ if either $\mathsf a=0$ and $\mathsf b=-2$, or
\[
\mathsf a>0\,,\quad-2\le\mathsf b<\mathsf a-1\quad\mbox{and}\quad\mathsf b\le2\,\mathsf a-2\,.
\]
If either $\mathsf a<0$ or $\mathsf b<-2$ or $\mathsf b>\min\{\mathsf a-1,2\,\mathsf a-2\}$ or $(\mathsf a,\mathsf b)=(1,0)$, then
\[
\inf_{\rho\,\in\mathcal X_1}\mathcal F_{\mathsf a,\mathsf b}[\rho]=-\infty\,.
\]
If $0\le\mathsf a<1$ and $\mathsf b=2\,\mathsf a-2$, then
\[
\mathcal C(\mathsf a,2\,\mathsf a-2)=-\,\log\(\frac{e\,\pi}{1-\mathsf a}\)\,.
\]
Moreover, if $\mathsf a>0$ there is no minimizer for $\mathcal C(\mathsf a,2\,\mathsf a-2)$.
\end{theorem}

\setlength{\unitlength}{1cm}
\begin{figure}[hb]
\begin{picture}(10,6){\includegraphics[width=10cm]{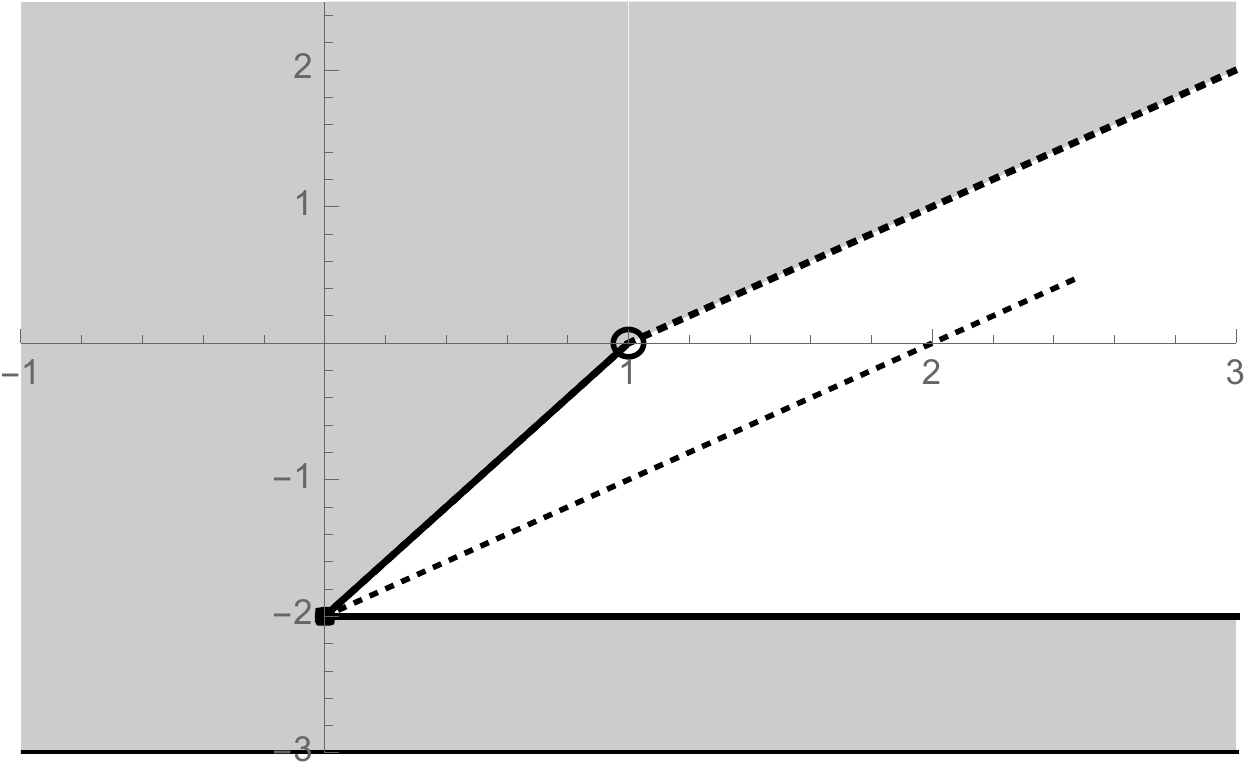}}
\put(0,0){$\mathsf a$}
\put(-7.3,6.4){$\mathsf b$}
\put(-2,4.5){\small$\mathsf b=\min\{\mathsf a-1,2\,\mathsf a-2\}$}
\put(-3.8,2.5){\small$\mathsf b=\mathsf a-2$}
\put(-6,4){$(1,0)$}
\put(-7,0.75){$(0,-2)$}
\end{picture}
\caption{\label{Fig0}\small\emph{White (resp.~grey) area corresponds to the domain in which~\eqref{Ineq:LogHLSimproved} holds for some finite constant $\mathcal C(\mathsf a,\mathsf b)$ (resp.~$\mathcal C(\mathsf a,\mathsf b)=-\infty$). We also know that $\mathcal C\big(\mathsf a,2\,(\mathsf a-1)\big)=-\,\log\({e\,\pi}/(1-\mathsf a)\)$ if $0\le\mathsf a<1$ and $\mathcal C(1,0)=-\infty$, while the boundedness from below of $\mathcal F_{\mathsf a,\mathsf b}$ is not known in the threshold case $\mathsf b=\mathsf a-1>0$. On the dotted half-line $\mathsf b=\mathsf a-2\ge-2$, optimality is achieved by $\rho_\star$ and Inequality~\eqref{Ineq:LogHLSimproved} corresponds to~\eqref{Ineq:LogHLS} with $\mathsf a=2\,\param$, $\mathsf b=2\,(\param-1)$, and $\tau\ge0$.}}
\end{figure}
The boundedness from below of $\mathcal F_{\mathsf a,\mathsf b}$ is unknown only in the case $\mathsf b=\mathsf a-1>0$. If $\mathsf b=2\,\mathsf a-2<0$, we do not only show the semi-boundedness of $\mathcal F_{\mathsf a,\mathsf b}$, but we actually compute the infimum $\mathcal C(\mathsf a,\mathsf b)$. The infimum is also known if $\mathsf b=\mathsf a-2\ge-2$ and in that case optimality is achieved by $\rho_\star$ according to~\eqref{Ineq:LogHLS}. Note that for $\mathsf a=0$, the inequality $\mathcal F_{\mathsf a,2\,(\mathsf a-1)}[\rho]\ge\mathcal C\big(\mathsf a,2\,(\mathsf a-1)\big)\,M$ is the sharp logarithmic Hardy-Littlewood-Sobolev inequality~\eqref{logHLS} and as $\mathsf a\to 1$ the infimum diverges to $-\infty$ consistently with the result of Lemma~\ref{Lem:LogHLSlim1}. For the convenience of the reader, we divide the proof of Theorem~\ref{Thm:FreeEnergy} in several intermediate result.
\begin{lemma}\label{Lem:LogHLSimproved} Inequality~\eqref{Ineq:LogHLSimproved} holds for some $\mathcal C(\mathsf a,\mathsf b)>-\infty$ if either $\mathsf a=0$ and $\mathsf b=-2$, or
\[
\mathsf a>0\quad\mbox{and}\quad-2\le\mathsf b<\min\{\mathsf a-1,2\,\mathsf a-2\}\,.
\]
\end{lemma}
The proof for $-2\le\mathsf b<0$ and $\mathsf a>1-\mathsf b/2$ follows from the case $\mathsf a=1-\mathsf b/2$, which is treated in Lemmas~\ref{lemma1} and~\ref{lemma2} below, but we give the argument here nevertheless, since it is simpler.
\begin{proof} The case $\mathsf a=0$ and $\mathsf b=-2$ corresponds to~\eqref{logHLS}. The case $\mathsf a=\eta>1$ and $\mathsf b=0$ is~\eqref{Ineq:LogHLS-ll}.

If $\mathsf b<0$, the condition $\mathsf b<2\,\mathsf a-2$ arises by combining~\eqref{logHLS} and~\eqref{Ineq:LogHLS-ll}, respectively multiplied by $-\mathsf b/2$ and $1+\mathsf b/2$, with $\mathsf a=(1+\mathsf b/2)\,\eta$ for any $\eta>1$. In that case,~\eqref{Ineq:LogHLSimproved} holds with
\begin{multline*}
\mathcal C(\mathsf a,\mathsf b)=M\,(1+\log\pi)\,\frac{\mathsf b}2+M\,\log\(\frac{\eta-1}\pi\)\(1+\frac{\mathsf b}2\)\\
=M\,(1+\log\pi)\,\frac{\mathsf b}2+M\,\log\(\frac{2\,\mathsf a-2-\mathsf b}{\pi\,(\mathsf b+2)}\)\frac{\mathsf b+2}2\,.
\end{multline*}
If $\mathsf b>0$, we sum~\eqref{Ineq:LogHLS-l} with a coefficient $\mathsf b/2$ and~\eqref{Ineq:LogHLS-ll} with coefficient $1$ and $\eta=\mathsf a-\mathsf b>1$. In that case,~\eqref{Ineq:LogHLSimproved} holds with
\[
\mathcal C(\mathsf a,\mathsf b)=M\,\frac{\mathsf b}2+M\,\log\(\frac{\eta-1}\pi\)=M\,\frac{\mathsf b}2+M\,\log\(\frac{\mathsf a-\mathsf b-1}\pi\)\,.
\]
\end{proof}
With $M=1$, notice that
\[
\mathcal F_{\mathsf a,\mathsf b}[\rho]=\irtwo{\rho\,\log\rho}+\mathsf a\irtwo{\log\(1+|x|^2\)\rho}+2\,\pi\,\mathsf b\irtwo{\rho\,(-\Delta)^{-1}\rho}\,.
\]
\begin{lemma}\label{Prop:Unbounded} If either $\mathsf a<0$ or $\mathsf b<-2$ or $\mathsf b>\min\{\mathsf a-1,2\,\mathsf a-2\}$ or $(\mathsf a,\mathsf b)=(1,0)$, then
\[
\inf_{\rho\,\in\mathcal X_1}\mathcal F_{\mathsf a,\mathsf b}[\rho]=-\infty\,.
\]
\end{lemma}
In Lemma~\ref{Prop:Unbounded}, there is no loss of generality in assuming that $M=1$. Under the assumptions on $(\mathsf a,\mathsf b)$ of Lemma~\ref{Prop:Unbounded}, Inequality~\eqref{Ineq:LogHLSimproved} does not hold for some $\mathcal C(\mathsf a,\mathsf b)>-\infty$. In that case, we shall simply write $\mathcal C(\mathsf a,\mathsf b)=-\infty$. See Fig.~\ref{Fig0}.

\begin{proof} For an arbitrary $\rho\in\mathcal X_1$, \emph{i.e.}, $\rho\in\mathrm L^1_+(\R^2)$ such that $\nrm\rho1=1$, let $\rho_{x_0}(x):=\rho(x-x_0)$. Since
\[
\irtwo{\log\(1+|x|^2\)\rho_{x_0}(x)}\sim2\,\log|x_0|\irtwo\rho\quad\mbox{as}\quad|x_0|\to+\infty
\]
and all other integrals are unchanged, the conclusion is straightforward if $\mathsf a<0$.

Assume now that $\rho\in\mathcal X_1$ is such that $\rho\,\log\rho$ and $\log\(1+|x|^2\)\rho$ are integrable, and let $\rho_\lambda(x)=\lambda^2\,\rho(\lambda\,x)$, for any $x\in\R^2$. We have
\begin{align*}
&\irtwo{\rho_\lambda\,\log\rho_\lambda}=\irtwo{\rho\,\log\rho}+2\,\log\lambda\,,\\
&\irtwo{\log\(1+|x|^2\)\rho_\lambda}=\irtwo{\log\(1+\lambda^{-2}\,|x|^2\)\rho}\,,\\
&\irtwo{\rho_\lambda\,(-\Delta)^{-1}\rho_\lambda}=\irtwo{\rho\,(-\Delta)^{-1}\rho}+\frac{\log\lambda}{2\,\pi}\,.
\end{align*}
As $\lambda\to+\infty$, we obtain that $\mathcal F_{\mathsf a,\mathsf b}[\rho_\lambda]\sim(\mathsf b+2)\,\log\lambda$, which proves our statement if $\mathsf b<-2$.

Assume additionally that $\rho(x) =0$ if $|x|\not\in[1,2]$. Since on any compact set of $\R^2\setminus\{0\}$, we have that $1+\lambda^{-2}\,|x|^2\sim\lambda^{-2}\,|x|^2$ as $\lambda\to0_+$ and deduce that
\[
\irtwo{\log\(1+|x|^2\)\rho_\lambda(x)}=\irtwo{\log\(1+\lambda^{-2}\,|x|^2\)\rho(x)}\sim-\,2\,\log\lambda\,.
\]
As $\lambda\to0_+$, we obtain that $\mathcal F_{\mathsf a,\mathsf b}[\rho_\lambda]\sim(\mathsf b+2-2\,\mathsf a)\,\log\lambda$, which proves our statement if $\mathsf b+2-2\,\mathsf a>0$.

Now, still assuming that $\rho(x) =0$ if $|x|\not\in[1,2]$, let
\[
\rho_{\varepsilon,\lambda}(x)=(1-\varepsilon)\,\rho(x)+\lambda^2\,\varepsilon\,\rho\(\lambda\,x\)
\]
with parameters $(\varepsilon,\lambda)\in(0,1)^2$. Using that the supports of $\rho$ and $\rho_{\lambda}$ decouple if $\lambda < 1/2$, we have, for any given $\varepsilon\in(0,1)$, as $\lambda\to0_+$
\[
\irtwo{\rho_{\varepsilon,\lambda}\,\log\rho_{\varepsilon,\lambda}}=\irtwo{\rho\,\log\rho}+\varepsilon\,\log\varepsilon+(1-\varepsilon)\,\log(1-\varepsilon)+2\,\varepsilon\,\log\lambda\,,
\]
\begin{multline*}
\irtwo{\log\(1+|x|^2\)\rho_{\varepsilon,\lambda}}=(1-\varepsilon)\irtwo{\log\(1+|x|^2\)\rho}+2\,\varepsilon\irtwo{\log|x|\,\rho}\\
-2\,\varepsilon\,\log\lambda+o\(\log\lambda\)\,,
\end{multline*}
\begin{multline*}
\irtwo{\rho_{\varepsilon,\lambda}\,(-\Delta)^{-1}\rho_{\varepsilon,\lambda}}=\(\varepsilon^2+(1-\varepsilon)^2\)\irtwo{\rho\,(-\Delta)^{-1}\rho}+\varepsilon^2\,\frac{\log\lambda}{2\,\pi}\\
-\frac{\varepsilon\,(1-\varepsilon)}\pi\irtwo{\log|x|\,\rho(x)}+\frac{\varepsilon\,(1-\varepsilon)}\pi\,\log\lambda+o\(\log\lambda\).
\end{multline*}
Thus,
\begin{multline*}
\irtwo{\rho_{\varepsilon,\lambda}\,\log\rho_{\varepsilon,\lambda}}+\mathsf a\irtwo{\log\(1+|x|^2\)\rho_{\varepsilon,\lambda}}+2\,\pi\,\mathsf b\irtwo{\rho_{\varepsilon,\lambda}\,(-\Delta)^{-1}\rho_{\varepsilon,\lambda}}\\
\sim2\,\varepsilon\(\(1-\frac\varepsilon2\)\mathsf b+1-\mathsf a\)\,\log\lambda\quad\mbox{as}\quad\lambda\to0_+\,.
\end{multline*}
This again proves our statement if $\mathsf b+1-\mathsf a>0$, because $\(1-\varepsilon/2\)\mathsf b+1-\mathsf a$ can be made positive for $\varepsilon>0$, small enough.
\end{proof}

The proof of Theorem \ref{Thm:FreeEnergy} in the case $\mathsf b=2\,(\mathsf a-1)\in[-2,0)$ is based on two ingredients exposed in Lemma~\ref{lemma1} and  Lemma~\ref{lemma2}. The first ingredient relates the minimization of $\mathcal F_{\mathsf a,2\,(\mathsf a-1)}$ to a simpler, scale-invariant minimization problem. Let
\[
\mathcal G_{\mathsf a}[\rho] := \int_{\R^2} \rho\,\log\rho\,dx + 2\,\mathsf a \int_{\R^2}\log |x|\,\rho\,dx + 2\,(\mathsf a-1) \iint_{\R^2\times\R^2} \rho(x)\,\log\frac{1}{|x-y|}\,\rho(y)\,dx\,dy
\]
and
\[
\mathcal K(\mathsf a) := \inf\left\{ \mathcal G_{\mathsf a}[\rho]\,:\,\rho\ge0 \,,\ \int_{\R^2} \rho\,dx = 1 \right\} .
\]
\begin{lemma}\label{lemma1} Let $0\le\mathsf a<1$. Then
\[
\mathcal C\big(\mathsf a,2\,(\mathsf a-1)\big) = \mathcal K(\mathsf a) \,.
\]
Moreover, if $\mathsf a>0$ there is no minimizer for $\mathcal C\big(\mathsf a,2\,(\mathsf a-1)\big)$.
\end{lemma}
\begin{proof} Since $\log\(1+|x|^2\)> 2\,\log|x|$, we immediately obtain $\mathcal F_{\mathsf a,2\,(\mathsf a-1)}[\rho]> \mathcal G_{\mathsf a}[\rho]$ whenever $\rho\not\equiv 0$ (and the functionals are finite). This implies that $\mathcal C\big(\mathsf a,2\,(\mathsf a-1)\big) \ge\mathcal K(\mathsf a)$ and that, if $\mathsf a>0$ and if $\mathcal F_{\mathsf a,2\,(\mathsf a-1)}$ has a minimizer, then $\mathcal C\big(\mathsf a,2\,(\mathsf a-1)\big) > \mathcal K(\mathsf a)$.

We show now the opposite inequality $\mathcal C\big(\mathsf a,2\,(\mathsf a-1)\big) \le\mathcal K(\mathsf a)$, which will complete the proof. Let $\sigma\ge0$ with $\int_{\R^2}\sigma\,dx = 1$ and with compact support not containing the origin. Consider $\rho_\lambda(x) = \lambda^{-2} \sigma(x/\lambda)$ with $\lambda\gg 1$. Then, as in the proof of Lemma~\ref{Prop:Unbounded},
\begin{multline*}
\mathcal F_{\mathsf a,\mathsf b}[\rho_\lambda] = (-2+2\,\mathsf a-\mathsf b)\log\lambda \\
+ \int_{\R^2} \sigma\,\log\sigma\,dx +\mathsf a\int_{\R^2}\log\(\lambda^{-2} + |x|^2\) \sigma(x)\,dx\\
 +\mathsf  b \iint_{\R^2\times\R^2} \sigma(x)\,\log\frac{1}{|x-y|}\,\sigma(y)\,dx\,dy \,.
\end{multline*}
If $\mathsf b=2\,(\mathsf a-1)$, then the coefficient of $\log\lambda$ vanishes and we obtain
\[
\mathcal C\big(\mathsf a,2\,(\mathsf a-1)\big) \le\liminf_{\lambda\to\infty} \mathcal F_{\mathsf a,2\,(\mathsf a-1)}[\rho_\lambda] = \mathcal G_{\mathsf a}[\sigma] \,.
\]
Taking the infimum over all $\sigma$ (and removing the support assumptions by an approximation argument), we obtain $\mathcal C\big(\mathsf a,2\,(\mathsf a-1)\big) \le\mathcal K(\mathsf a)$, as claimed.	
\end{proof}

\begin{lemma}\label{lemma2} Let $0\le\mathsf a<1$. Then
\[
\mathcal K(\mathsf a) = -\log\(\frac{e\,\pi}{1-\mathsf a}\) \,.
\]
For $\mathsf a>0$ the infimum $\mathcal K(\mathsf a)$ is achieved if and only if, for some $\lambda>0$,
\[
\rho(x) = \frac{1-\mathsf a}{\pi}\, \frac{\lambda^2}{|x|^{2\,\mathsf a}\(\lambda^2 + |x|^{2\,(1-\mathsf a)}\)^2} \,.
\]
\end{lemma}
The idea of the proof is to apply a change of variables and to reduce the result to the case $\mathsf a=0$.
\begin{proof}[Proof of Lemma \ref{lemma2}]
By symmetric decreasing rearrangement it suffices to bound $\mathcal G_{\mathsf a}[\rho]$ from below for radial decreasing $\rho$. In fact, in the following we only use that $\rho$ is radial, and we use this in order to apply Newton's theorem. We set $\tilde\rho(x) := |x|^{2\,\mathsf a}\,\rho(x)$ and then we define a radial function $\tau$ on $\R^2$ by $\tau(z) = \tilde\rho\(|z|^{1/(1-\mathsf a)}\)$ (with an obvious abuse of notation for the radial function $\tilde\rho$). We have
\begin{align*}
\int_{\R^2} \tau(z)\,dz & = 2\,\pi \int_0^\infty \tilde\rho\(r^{1/(1-\mathsf a)}\) r\,dr = 2\,\pi\,(1-\mathsf a) \int_0^\infty \tilde\rho(s)\,s^{1-2\,\mathsf a}\,ds \\
& = 2\,\pi\,(1-\mathsf a) \int_0^\infty \rho(s)\,s\,ds = (1-\mathsf a) \int_{\R^2}\rho(x)\,dx = 1-\mathsf a \,.
\end{align*}
Moreover, by a similar computation,
\begin{align*}
\int_{\R^2} \rho\,\log\rho\,dx + 2\,\mathsf a \int_{\R^2}\log|x|\,\rho\,dx & = \int_{\R^2} \tilde\rho\,\log\tilde\rho\,|x|^{-2\,\mathsf a}\,dx = 2\,\pi \int_0^\infty \tilde\rho(s)\,\log\tilde\rho(s)\,s^{1-2\,\mathsf a}\,ds \\
& = \frac{2\,\pi}{1-\mathsf a} \int_0^\infty \tau(r)\,\log\tau(r)\,r\,dr = \frac{1}{1-\mathsf a} \int_{\R^2} \tau(z)\,\log\tau(z)\,dz \,.		
\end{align*}
Finally, by Newton's theorem,
\begin{align*}
\iint_{\R^2\times\R^2} \rho(x)\,&\log\frac{1}{|x-y|}\,\rho(y)\,dx\,dy\\
& = \iint_{\R^2\times\R^2} \rho(x) \min\left\{\log\frac{1}{|x|},\log\frac{1}{|y|} \right\} \rho(y)\,dx\,dy \\
& = \iint_{\R^2\times\R^2} \tilde\rho(x) \min\left\{\log\frac{1}{|x|},\log\frac{1}{|y|} \right\} \tilde\rho(y)\,\frac{dx}{|x|^{2\,\mathsf a}}\,\frac{dy}{|y|^{2\,\mathsf a}} \\
& = (2\,\pi)^2 \int_0^\infty \int_0^\infty \tilde\rho(s) \min\left\{\log\frac{1}{s},\log\frac{1}{s'} \right\} \tilde\rho(s')\,\frac{ds}{s^{2\,\mathsf a-1}}\,\frac{ds'}{(s')^{2\,\mathsf a-1}} \\
& = \frac{(2\,\pi)^2}{(1-\mathsf a)^3} \int_0^\infty \int_0^\infty \tau(r) \min\left\{\log\frac{1}{r},\log\frac{1}{r'} \right\} \tau(r')\,r\,dr\,r'dr'	\\
& = \frac{1}{(1-\mathsf a)^3} \iint_{\R^2\times\R^2} \tau(z)\,\log\frac{1}{|z-w|}\,\tau(w)\,dz\,dw \,.
\end{align*}
To summarize, we have
\[
\mathcal G_{\mathsf a}[\rho] = \frac{1}{1-\mathsf a} \int_{\R^2} \tau(z)\,\log\tau(z)\,dz - \frac{2}{(1-\mathsf a)^2} \iint_{\R^2\times\R^2} \tau(z)\,\log\frac{1}{|z-w|}\,\tau(w)\,dz\,dw \,.
\]
By the logarithmic Hardy-Littlewood-Sobolev inequality~\eqref{logHLS}, taking the normalization of $\tau$ into account, we deduce that
\[
\mathcal G_{\mathsf a}[\rho] \ge-\log\(\frac{e\,\pi}{1-\mathsf a}\)
\]
with equality if and only if, for some $\lambda>0$,
\[
\tau(z) = \frac{1-\mathsf a}{\pi}\, \frac{\lambda^2}{(\lambda^2+|z|^2)^2} \,.
\]
Translating this in terms of $\rho$, we obtain the claim of the lemma.	
\end{proof}

\subsection{Additional remarks on the free energy and some open questions}\label{Sec:FEopen}

In Lemma~\ref{Lem:LogHLSimproved}, Inequality~\eqref{Ineq:LogHLSimproved} holds for some finite constant $\mathcal C(\mathsf a,\mathsf b)$ if $(\mathsf a,\mathsf b)=(0,-2)$ . We also know from Lemma~\ref{Lem:LogHLSlim1} that $\lim_{\mathsf a\to1_+}\mathcal C(\mathsf a,0)=-\infty$. If $\mathsf b=\mathsf a-1>0$, it is so far open to decide whether~\eqref{Ineq:LogHLSimproved} holds for some $\mathcal C(\mathsf a,\mathsf b)>-\infty$. See Fig.~\ref{Fig0}.

The free energy $\mathcal F_{\mathsf a,\mathsf b}[\rho]$ is a natural Lyapunov functional for the drift-diffusion equation
\be{DDE}
\frac{\partial\rho}{\partial t}=\Delta\rho+\nabla\cdot\Big(\rho\,\big(\mathsf a\,\nabla V+4\,\pi\,\tfrac{\mathsf b}M\,\nabla W\big)\Big)\,,\quad W=(-\Delta)^{-1}\rho\,.
\ee
Indeed we can write that $\Delta\rho=\nabla\cdot\(\rho\,\nabla\log\rho\)$ so that, for any smooth and sufficiently decreasing function $\rho$ solving~\eqref{DDE}, we obtain using an integration by parts that
\[
\frac d{dt}\mathcal F_{\mathsf a,\mathsf b}[\rho(t,\cdot)]=-\irtwo{\rho\,\left|\nabla\log\rho+\mathsf a\,\nabla V+4\,\pi\,\tfrac{\mathsf b}M\,\nabla W\right|^2}\,.
\]
Concerning the the long time behavior of the solution of~\eqref{DDE}, we expect that $\mathcal F_{\mathsf a,\mathsf b}[\rho(t,\cdot)]$ converges to $\mathcal C(\mathsf a,\mathsf b)$ as $t\to+\infty$ by analogy, \emph{e.g.}, with the Keller-Segel system (see~\cite[Section~4]{MR2226917}), but it is an open question to deduce global decay rates of $\mathcal F_{\mathsf a,\mathsf b}[\rho(t,\cdot)]$, for instance in a restricted class of solutions of~\eqref{DDE}, or even asymptotic decay rates as in~\cite{MR3196188}. Another issue is to understand the counterpart on $\mathbb S^2$ of the results on $\R^2$ using the inverse stereographic projection, as in~\cite{MR1143664,MR2915466,MR3227280}.

For any $M>0$, the boundedness from below of
\[
\mathcal F_{\mathsf a,\mathsf b}^{\,\mathsf c}[\rho]:=\mathsf a\irtwo{\!\log\(1+|x|^2\)\rho}-\frac{\mathsf b}M\iint_{\R^2\times\R^2}\!\rho(x)\,\rho(y)\,\log|x-y|\,dx\,dy+\mathsf c\irtwo{\!\rho\,\log\(\frac\rho M\)}
\]
on the set $\mathcal X_M$ arises for any $\mathsf c>0$ as a straightforward consequence of Lemma~\ref{Lem:LogHLSimproved} under the obvious condition $-\,2\,\mathsf c\le\mathsf b<\min\{\mathsf a-\mathsf c,2\,\mathsf a-2\,\mathsf c\}$, by homogeneity. The case $\mathsf c=0$ is covered by Lemma~\ref{Lem:LogHLSlim}. It is therefore a natural question to inquire what happens if $\mathsf c<0$.
\begin{proposition}\label{Prop:NegUnbounded} For any $(\mathsf a,\mathsf b)\in\R^2$ and $M>0$, with the above notations, if $\mathsf c<0$, then
\[
\inf_{\rho\,\in\mathcal X_M}\mathcal F_{\mathsf a,\mathsf b}^{\,\mathsf c}[\rho]=-\infty\,.
\]
\end{proposition}
\begin{proof} The key point of the proof is that $\rho\mapsto\mathsf c\irtwo{\rho\,\log\rho}$ with $\mathsf c<0$ is a concave functional. Let $\rho\in\mathcal X_1$ be a function supported in the unit ball. For any $\varepsilon\in(0,1/4)$ and $n\in\mathbb N\setminus\{0\}$, let
\[
R_{\varepsilon,n}(x):=\frac1{n^2}\sum_{k,\ell=1}^n\varepsilon^{-2}\,\rho\(\varepsilon^{-1}\,\big(x-(k,\ell)\big)\)\quad\forall\,x\in\R^2\,.
\]
In order to investigate the limits $\varepsilon\to0_+$ and $n\to+\infty$, we compute
\begin{align*}
&\irtwo{R_{\varepsilon,n}\,\log R_{\varepsilon,n}}=-\,\log\(n^2\,\varepsilon^2\)+\irtwo{\rho\,\log\rho}=-\,2\,\log\(n\,\varepsilon\)+O(1)\,,\\
&\irtwo{\log\(1+|x|^2\)R_{\varepsilon,n}}\lesssim\log\(1+2\,n^2\)=2\,\log n\,\big(1+o(1)\big)\,,\\
&\left|\iint_{\R^2\times\R^2}R_{\varepsilon,n}(x)\,R_{\varepsilon,n}(y)\,\log|x-y|\,dx\,dy\right|\lesssim\frac{|\log\varepsilon|}{n^2}+\frac{\log\(1+2\,n^2\)}{2\,n^2}\\
&\hspace*{10cm}=\frac{|\log(\varepsilon/n)|}{n^2}\,\big(1+o(1)\big)\,.
\end{align*}
With the choice $\varepsilon=n^{-A}$ for some $A>0$ large enough, we find that $\mathsf c\irtwo{R_{\varepsilon,n}\,\log R_{\varepsilon,n}}\sim(A-1)\,|\mathsf c|\,\log n\to-\infty$ as $n\to+\infty$ and this term dominates the other ones. This concludes the proof.
\end{proof}

\section{Logarithmic interpolation inequalities and Schr\"odinger energy estimates}\label{Sec:Schrodinger}

We are now going to study the Schr\"odinger energy $\mathcal E$ defined by~\eqref{SchEnergy}. As we shall see, the kinetic energy $\irtwo{|\nabla u|^2}$ completely changes the picture and considering $\mathsf c<0$ makes sense.

\subsection{A new logarithmic interpolation inequality}

Here we combine logarithmic Hardy-Littlewood-Sobolev inequalities with the logarithmic Sobolev inequality to produce a new logarithmic interpolation inequality. This new inequality is more directly connected with the Schr\"odinger-Poisson system~(SP).

In dimension $d=2$, with the Gaussian measure defined as $d\mu=\mu(x)\,dx$ where $\mu(x)=(2\,\pi)^{-1}\,\exp(-|x|^2/2)$, the \emph{Gaussian logarithmic Sobolev inequality} reads
\be{Ineq:LogSobGaussian}
\irtwomu{|\nabla v|^2}\ge\frac12\irtwomu{|v|^2\,\log|v|^2}
\ee
for any function $v\in\mathrm H^1(\R^2,d\mu)$ such that $\irtwomu{|v|^2}=1$, and there is equality if and only if $v\equiv1$ (see~\cite[Theorem~4]{MR1132315}). With $u=v\,\sqrt\mu$, it is a classical fact that Inequality~\eqref{Ineq:LogSobGaussian} is equivalent to the standard \emph{Euclidean logarithmic Sobolev inequality} established in~\cite{Gross75} (also see~\cite{Federbush} for an earlier related result) which can be written in dimension $d=2$ as
\be{Ineq:LogSobGaussianEuclidean}
\irtwo{|\nabla u|^2}\ge\frac12\irtwo{|u|^2\,\log\(\frac{|u|^2}{\nrm u2^2}\)}+\frac12\,\log\(2\,\pi\,e^2\)\nrm u2^2
\ee
for any function $u\in\mathrm H^1(\R^2,dx)$. This inequality is not invariant under scaling. By applying~\eqref{Ineq:LogSobGaussianEuclidean} to the scaled function $u_\lambda(x)=\lambda\,u(\lambda\,x)$, we obtain
\be{Ineq:LogSobGaussianEuclideanParam}
\lambda^2\,\nrm{\nabla u}2^2-\log\lambda\,\nrm u2^2\ge\frac12\irtwo{|u|^2\,\log\(\frac{|u|^2}{\nrm u2^2}\)}+\frac12\,\log\(2\,\pi\,e^2\)\nrm u2^2
\ee
for any $\lambda>0$. The scaling parameter $\lambda$ can be optimized in order to obtain the \emph{Euclidean logarithmic Sobolev inequality in scale invariant form}
\be{Ineq:LogSobEuclideanWeissler}
\nrm u2^2\,\log\(\frac1{\pi\,e}\,\frac{\nrm{\nabla u}2^2}{\nrm u2^2}\)\ge\irtwo{|u|^2\,\log\(\frac{|u|^2}{\nrm u2^2}\)}
\ee
for any function $u\in\mathrm H^1(\R^2,dx)$, that can be found in~\cite[Theorem~2]{MR479373},~\cite[Inequality~(2.3)]{MR0109101},~\cite[Appendix~B]{MR823597} or~\cite[Inequality~(26)]{MR1132315}. See~\cite{MR1768665,MR3493423} for further references and consequences. Of course,~\eqref{Ineq:LogSobGaussianEuclidean} can be deduced from~\eqref{Ineq:LogSobEuclideanWeissler}, so that~\eqref{Ineq:LogSobGaussian},~\eqref{Ineq:LogSobGaussianEuclidean} and~\eqref{Ineq:LogSobEuclideanWeissler} are equivalent, and none of these inequalities is limited to $d=2$, but constants in~\eqref{Ineq:LogSobGaussianEuclidean} and~\eqref{Ineq:LogSobEuclideanWeissler} have to be adapted to the dimension if $d\neq2$.

It is possible to combine~\eqref{logHLS} and~\eqref{Ineq:LogSobGaussianEuclidean} with $\rho=|u|^2$ into
\be{Ineq}
\irtwo{|\nabla u|^2}\ge\frac{2\,\pi}{\nrm u2^2}\irtwo{|u|^2\,(-\Delta)^{-1}|u|^2}+\frac12\,\log(2\,e)\,\nrm u2^2
\ee
where
\[
2\,\pi\irtwo{|u|^2\,(-\Delta)^{-1}|u|^2}=-\iint_{\R^2\times\R^2}|u(x)|^2\,\log|x-y|\,|u(y)|^2\,dy\,.
\]
By applying~\eqref{Ineq} to the scaled function $u_\lambda(x)=\lambda\,u(\lambda\,x)$, we obtain that
\be{Ineq1Param}
\lambda^2\irtwo{|\nabla u|^2}-\nrm u2^2\,\log\lambda\ge\frac{2\,\pi}{\nrm u2^2}\irtwo{|u|^2\,(-\Delta)^{-1}|u|^2}+\frac12\,\log(2\,e)\,\nrm u2^2
\ee
for any $\lambda>0$. By optimizing on $\lambda$, we obtain the following scale invariant inequality.
\begin{proposition}\label{Prop:Ineq} For any function $u\in\mathrm H^1(\R^2)$, we have
\be{ScaledIneq}
2\,\pi\irtwo{|u|^2\,(-\Delta)^{-1}|u|^2}\le\nrm u2^4\,\log\(\frac{\nrm{\nabla u}2}{\nrm u2}\)\,.
\ee
\end{proposition}
Since~\eqref{logHLS} and~\eqref{Ineq:LogSobGaussianEuclidean} admit incompatible optimal functions, respectively the function $\rho=\rho_\star$ given by~\eqref{rhostar} and the Gaussian function $u(x)=(2\,\pi)^{-1/2}\,\sqrt M\,e^{-|x|^2/4}=\sqrt{M\,\mu(x)}$, up to multiplications by a constant, scalings and translations, equality is not achieved in~\eqref{ScaledIneq} by a function $u\in\mathrm H^1(\R^2)$.

\subsection{Interpolations inequalities in higher dimensions}\label{Sec:IneqHigher}

For comparison, let us briefly consider the case of higher dimensions, that is, the case of the Euclidean space~$\R^d$ with $d\ge3$. We can refer for instance to~\cite{MR2826402} for more detailed considerations on scalings in absence of an external potential. The Gagliardo-Niren\-berg inequality
\be{GN}
\mathcal C_{\rm GN}\,\nrm{\nabla u}2^\vartheta\,\nrm u2^{1-\vartheta}\ge\nrm up\quad\forall\,u\in\mathrm H^1(\R^d)
\ee
holds with $\theta=d\,\frac{p-2}{2\,p}$ for any $p\in(2,2^*]$, where $2^*=\frac{2\,d}{d-2}$ is the critical Sobolev exponent. Optimality is attained by the so-called \emph{Lommel functions}, which are radial functions according to, \emph{e.g.},~\cite{MR691044}, and are defined by the Euler-Lagrange but have no explicit formulation in terms of the usual special functions: see~\cite{Von_Lommel_1875,Lommel_1880}. This can be combined with the \emph{critical Hardy-Littlewood-Sobolev inequality},
\be{HLS}
\frac1{(d-2)\,|\mathbb S^{d-1}|}\iint_{\R^d\times\R^d}\frac{\rho(x)\,\rho(y)}{|x-y|^{d-2}}\,dx\,dy=\ird{\rho\,(-\Delta)^{-1}\rho}\le\mathcal C_{\rm{HLS}}\(\ird{|\rho|^\frac{2\,d}{d+2}}\)^{1+\frac2d}
\ee
for any function $\rho\in\mathrm L^\frac{2\,d}{d+2}(\R^d)$, to establish for $\rho=|u|^2$ that
\be{GN-HLS}
\mathcal C_d\ird{|u|^2\,(-\Delta)^{-1}|u|^2}\le\nrm{\nabla u}2^{d-2}\,\nrm u2^{6-d}\quad\forall\,u\in\mathrm H^1(\R^d)\,,
\ee
under the condition that $\frac{4\,d}{d+2}\le\frac{2\,d}{d-2}$, that is, for
\[
3\le d\le6\,.
\]
Let us notice that the inequality is \emph{critical} if $d=6$ in the sense that $\int_{\R^6}|u|^2\,(-\Delta)^{-1}|u|^2\,dx$ and $\(\int_{\R^6}|\nabla u|^2\,dx\)^2$ have the same homogeneity and scaling invariance, which is a standard source of loss of compactness along an arbitrary minimizing sequence satisfying a given $\nrm u2$ constraint. From~\eqref{GN} and~\eqref{HLS}, we find out that
\[
\mathcal C_d\ge\mathcal C_{\rm{GN}}^{-4}\,\mathcal C_{\rm{HLS}}^{-1}\,.
\]
The above estimate is strict because optimal functions do not coincide in~\eqref{GN} and~\eqref{HLS} if $3\le d\le5$. In dimension $d=6$, we have that $\mathcal C_6=\mathcal C_{\rm{GN}}^{-4}\,\mathcal C_{\rm{HLS}}^{-1}$ is sharp, with equality in~\eqref{GN-HLS} achieved by the Aubin-Talenti function $x\mapsto(1+|x|^2)^{-2}$.

\subsection{Bounds on the Schr\"odinger energy}\label{Sec:SchrodingerBounds}

Let $\gamma_+:=\max\{\gamma,0\}$ and consider $\mathcal E$ as in~\eqref{SchEnergy}.
\begin{theorem}\label{Thm:E} Let $\alpha$, $\beta$, $\gamma$ be real parameters and assume that $M>0$. Then
\begin{enumerate}
\item[(i)] $\mathcal E$ is not bounded from below on $\mathcal H_M$ if one of the following conditions is satisfied:
\begin{align*}
\mathrm{(a)}\quad&\alpha<0\,,\\
\mathrm{(b)}\quad&\alpha\ge0\quad\mbox{and}\quad M\,\beta>\min\big\{2\,\alpha-\gamma,4\,\alpha-2\,\gamma\big\}\,.
\end{align*}
\item[(ii)] $\mathcal E$ is bounded from below on $\mathcal H_M$ if either $\alpha=0$, $\beta \leq 0$ and $M\,\beta+2\,\gamma\le0$, or $\alpha>0$ and one of the following conditions is satisfied:
\begin{align*}
\mathrm{(a)}\quad&\gamma\le0\quad\mbox{and}\quad M\,\beta \leq 2\,\alpha\,,\\
\mathrm{(b)}\quad&\gamma>0\,,\quad M\,\beta\le4\,\alpha-2\,\gamma
\quad\mbox{and}\quad M\,\beta<2\,\alpha-\gamma\,.
\end{align*}
\end{enumerate}
\end{theorem}
Two cases covered by Theorem~\ref{Thm:E} are shown in Fig.~\ref{Fig1}.
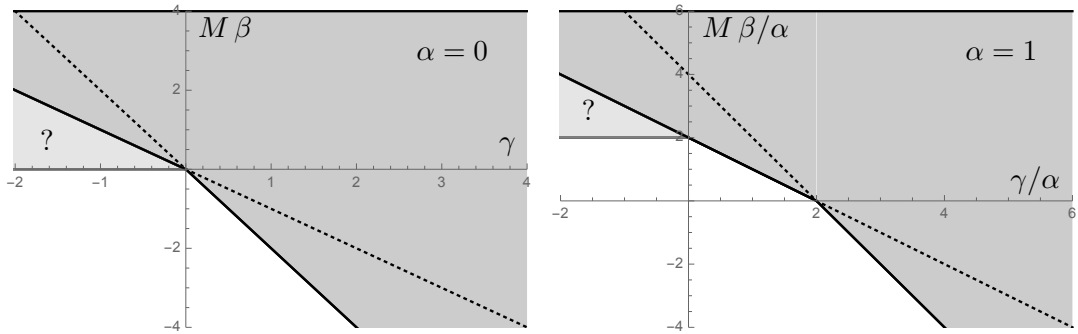
\begin{figure}[ht]
\begin{picture}(14.25,5){\includegraphics[width=7cm]{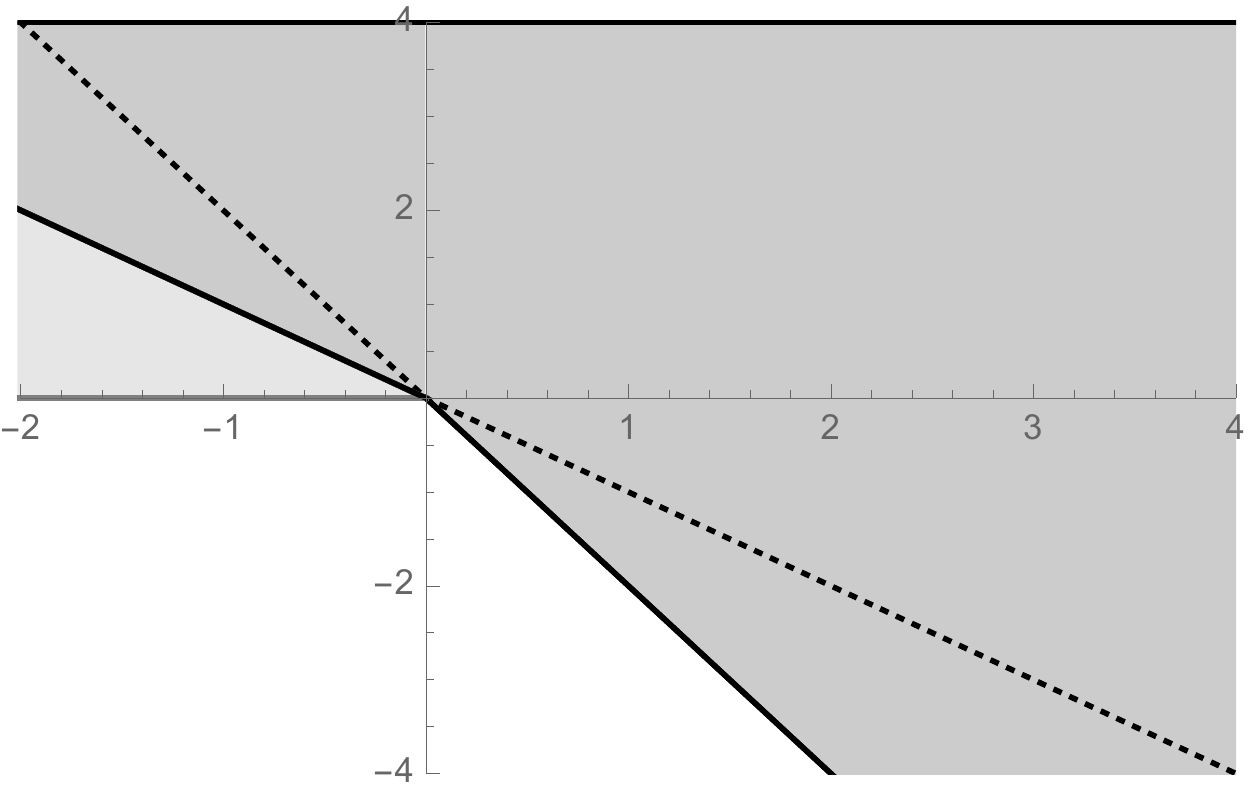}\hspace*{0.25cm}\includegraphics[width=7cm]{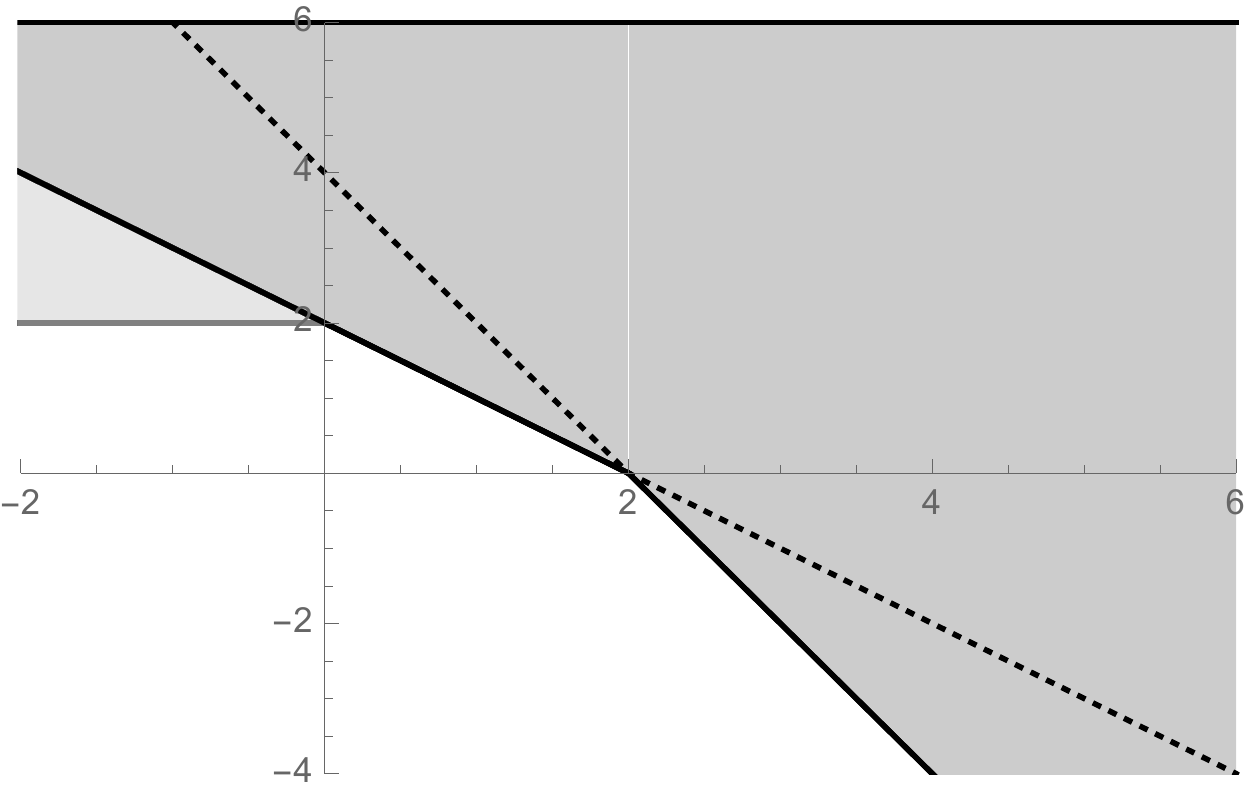}}
\put(-7.7,2.5){$\gamma$}
\put(-0.9,2){$\gamma/\alpha$}
\put(-13.8,2.5){$?$}
\put(-6.6,2.9){$?$}
\put(-8.8,3.7){$\alpha=0$}
\put(-1.5,3.7){$\alpha=1$}
\put(-11.7,4){$M\,\beta$}
\put(-5,4){$M\,\beta/\alpha$}
\end{picture}
\caption{\label{Fig1}\small\emph{White (resp.~dark grey) area corresponds to the domain in which~$\mathcal E$ is bounded (resp.~unbounded) from below with $\alpha=0$ on the left and $\alpha=1$ on the right. Whether $\mathcal E$ is bounded in the light grey domain or not is open so far.}}
\end{figure}
\begin{proof} Let us start by the proof of (i), \emph{i.e.}, the cases for which $\inf\{\mathcal E[u]\,:\,u\in\mathcal H_M\}=-\infty$. Case (a) corresponds to $\alpha<0$ and can be dealt with using translations as in the proof of Lemma~\ref{Prop:Unbounded}: $\lim_{|x_0|\to+\infty}\mathcal E[u(\cdot-x_0)]=-\infty$. Next let $u_\lambda(x):=\lambda\,u(\lambda\,x)$ and notice that
\[
\irtwo{|\nabla u_\lambda|^2}=\lambda^2\,\irtwo{|\nabla u|^2}=o(\log\lambda)\quad\mbox{as}\quad\lambda\to0_+\,,
\]
so that, with $\rho_\lambda=|u_\lambda|^2$,
\[
\mathcal E[u_\lambda]\sim2\,\alpha\irtwo{\log\(1+|x|^2\)\rho_\lambda}+2\,\pi\,\beta\irtwo{\rho_\lambda\,(-\Delta)^{-1}\rho_\lambda}+\gamma\irtwo{\rho_\lambda\,\log\rho_\lambda}\,.
\]
By arguing as in Lemma~\ref{Prop:Unbounded}, we obtain that $\lim_{\lambda\to0_+}\mathcal E[u_\lambda]=-\infty$ in case (b).

Concerning (ii), the boundedness from below of $\mathcal E$ is as follows. From~\eqref{Ineq:LogSobGaussianEuclideanParam} and~\eqref{Ineq1Param}, we learn that
\begin{equation}\label{Ine-1}
\irtwo{|\nabla u|^2}\ge\frac1{2\,\lambda_1^2}\irtwo{|u|^2\,\log\(\frac{|u|^2}M\)}+\frac{\log\(2\,\pi\,e^2\,\lambda_1^2\)}{2\,\lambda_1^2}\,M
\end{equation}
and
\begin{equation}\label{Ine-2}
\irtwo{|\nabla u|^2}\ge\frac{2\,\pi}{M\,\lambda_2^2}\irtwo{|u|^2\,(-\Delta)^{-1}|u|^2}+\frac{\log\(2\,e\,\lambda_2^2\)}{2\,\lambda_2^2}\,M
\end{equation}
with $\nrm u2^2=M$. Here $\lambda_1$ and $\lambda_2$ are two arbitrary positive parameters. Let us distinguish various cases:
\begin{enumerate}
\item If $\alpha=0$, $\beta\le0$ and $\gamma\le0$, the boundedness from below of $\mathcal E$ is a direct consequence of~\eqref{Ine-1} and~\eqref{Ine-2}. The case $\alpha=0$, $\beta<0$ and $\gamma>0$ can be reduced to the case $\alpha=0$ and $\gamma=0$ using~\eqref{logHLS} if $M\,\beta+2\,\gamma\le0$. 
\item If either $\alpha>0$, $\beta\le0$ and $\gamma\le0$, or $\alpha>0$, $\beta>0$,  $\gamma\le0$ and  $M\,\beta+2\,\gamma\le0$, we conclude as above.
\item If $\alpha>0$, $\beta>0$ and $\gamma\le0$, the boundedness from below is a direct consequence of Lemma~\ref{Lem:LogHLSlim} if  $M\,\beta-2\,\alpha\le0$.
\item If $\alpha >0$, $\gamma >0$ and $M\,\beta+2\,\gamma\ge0$, we notice that $\mathcal E[u]\ge\gamma\,\mathcal F_{\mathsf a,\mathsf b}\big[|u|^2\big]$ with $\mathsf a=2\,\alpha/\gamma$ and $\mathsf b=M\,\beta/\gamma$. The result of Lemma~\ref{Lem:LogHLSimproved} applies and the condition $\mathsf b<\min\{\mathsf a-1,2\,\mathsf a-2\}$ can be rewritten as $M\,\beta<\min\big\{2\,\alpha-\gamma,4\,\alpha-2\,\gamma\big\}$. The case $M\,\beta=4\,\alpha-2\,\gamma$ corresponds to $\mathsf b=2\,\mathsf a-2$ and it is covered by Lemmas~\ref{lemma1} and~\ref{lemma2}.
\item If $\alpha >0$, $\gamma >0$ and $M\,\beta+2\,\gamma<0$, we conclude by observing that
\[
\mathcal E[u]\ge\gamma\,\mathcal F_{\mathsf a,-2}\big[\,|u|^2\,\big]+\irtwo{|\nabla u|^2}+ \frac{2\,\pi}M\,(M\,\beta+2\,\gamma)\irtwo{|u|^2\,(-\Delta)^{-1}|u|^2}\,,
\]
where, because $M\,\beta+2\,\gamma<0$, the sum of the last two terms is bounded from below in view~\eqref{Ine-2} and where Lemma~\ref{Lem:LogHLSimproved} guarantees that $\mathcal F_{\mathsf a,-2}\big[\,|u|^2\,\big]$ is bounded from below.
\end{enumerate}
\nc\end{proof}

\noindent{\bf Acknowledgments:} Partial support through the French National Research Agency grant EFI ANR-17-CE40-0030 (J.D.), the US National Science Foundation grants DMS-1363432 and DMS-1954995 (R.L.F.) and the German Research Foundation DFG grant EXC-2111 – 390814868 (R.L.F.) is acknowledged. J.D.~and L.J.~address some special thanks to the organizers of the conference \emph{Nonlinear days in Alghero}, (September 16-20, 2019) where key results of this paper have been established.\\\noindent{\scriptsize\copyright\,2021 by the authors. This paper may be reproduced, in its entirety, for non-commercial purposes.}
\bibliographystyle{siamplain}\small
\bibliography{References}
\end{document}